\documentclass[11pt,a4paper]{amsart}

\usepackage[T1]{fontenc}
\usepackage[utf8]{inputenc}
\usepackage{lmodern}
\usepackage{amsmath,amssymb,amsthm,mathtools}
\usepackage[a4paper,textwidth=160mm,textheight=240mm,centering]{geometry}
\usepackage{microtype}
\usepackage[hidelinks]{hyperref}
\usepackage{enumitem}
\numberwithin{equation}{section}
\newtheorem{theorem}{Theorem}[section]
\newtheorem{proposition}[theorem]{Proposition}
\newtheorem{lemma}[theorem]{Lemma}
\newtheorem{corollary}[theorem]{Corollary}
\theoremstyle{remark}
\newtheorem{remark}[theorem]{Remark}

\newcommand{\C}{\mathbb C}
\newcommand{\R}{\mathbb R}
\newcommand{\Z}{\mathbb Z}
\newcommand{\F}{\mathbb F}
\newcommand{\E}{\mathbb E}
\newcommand{\Prob}{\mathbb P}
\newcommand{\calS}{\mathcal S}
\newcommand{\ee}[1]{e\left(#1\right)}
\newcommand{\norm}[1]{\lVert#1\rVert_\infty}
\DeclareMathOperator{\Rea}{Re}
\renewcommand{\leq}{\leqslant}
\renewcommand{\geq}{\geqslant}
\setlist[enumerate]{label=\textup{(\roman*)},leftmargin=*,itemsep=3pt}

\title[Space-filling curves and small-ball probabilities]
{Space-filling curves and small-ball probabilities for the Kloosterman process}

\author{Danyil Kyselov}
\address{Odessa I.~I. Mechnikov National University,
2 Dvoryanskaya Street, 65086 Odessa, Ukraine}
\email{daniil.kiselev@stud.onu.edu.ua}

\date{October 15, 2026}
\hypersetup{pdftitle={Space-filling curves and small-ball probabilities for the Kloosterman process},
pdfauthor={Danyil Kyselov},
pdfsubject={Explicit curves and quantitative approximation in the support of a random Fourier series}}

\begin{document}

\begin{abstract}
We construct explicitly a space-filling curve in the support of the limiting
Kloosterman process, using the planar partial-sum curve of the
Rudin--Shapiro sequence. The qualitative existence statement was previously
attributed to J.~Bober by Kowalski.
The resulting curve is a closed, $1/2$-H\"older
continuous loop whose image contains a specified triangle of area $1/128$.
Its Fourier coefficients are purely imaginary and satisfy
$|\widehat f(h)|\leq(2\pi|h|)^{-1}$ for $h\ne0$.
We also prove a uniform lower bound for small balls centred at closed
$1/2$-H\"older curves with a fixed coefficient bound, of the form
$\exp\{-C\varepsilon^{-2}\log^2(2/\varepsilon)\}$.
Known functional convergence then gives a quantitative lower bound for the
limiting proportion of two-parameter Kloosterman paths approximating the
constructed loop.
\end{abstract}

\maketitle

\section{Introduction}

Write $e(u)=\exp(2\pi iu)$. Let $(X_h)_{h\in\Z}$ be independent random
variables with the Sato--Tate distribution
\[
 d\mu_{\mathrm{ST}}(x)=\frac{\sqrt{4-x^2}}{2\pi}
 \boldsymbol 1_{[-2,2]}(x)\,dx.
\]
Kowalski and Sawin \cite{KSpaths} introduced the continuous random process
\begin{equation}\label{eq:process}
 K(t)=tX_0+\sum_{h\ne0}\frac{e(ht)-1}{2\pi ih}X_h,
 \qquad 0\leq t\leq1,
\end{equation}
where the symmetric partial sums converge almost surely uniformly.
Let $\calS$ denote the support of its law in $C([0,1],\C)$, equipped with
the uniform norm. By \cite[Theorem~1.1]{KSsupport}, a continuous function
$g$ belongs to $\calS$ if and only if $g(0)=0$, $g(1)\in[-2,2]$, and
$u(t)=g(t)-tg(1)$ satisfies
\begin{equation}\label{eq:support}
 \widehat u(h)\in i\R,\qquad
 |\widehat u(h)|\leq\frac1{\pi|h|}\quad(h\in\Z\setminus\{0\}),
 \qquad
 \widehat u(h)=\int_0^1u(t)e(-ht)\,dt.
\end{equation}

Kowalski and Sawin \cite[Introduction, question~(1)]{KSsupport}
asked whether $\calS$ contains a space-filling curve. A positive answer
was subsequently attributed to J.~Bober by Kowalski
\cite[Section~6.3, p.~110 of the author version]{KowalskiBook}.
More specifically, Kowalski's January~2020 lecture slides
\cite[slide~49]{KowalskiSlides} state that the support contains a curve
whose image contains a square, and refer to the Heilbronn Conference~2019.
The published bibliography of \cite{KowalskiBook} lists a related
three-author preprint of Bober, Kowalski and Sawin \cite{BKS}.
We have not located a publicly available text of that preprint and do not
claim priority for the qualitative existence statement.

This paper gives a direct construction answering the space-filling support
question, together with explicit geometric and Fourier bounds and a
uniform small-ball lower bound. The construction produces a closed
$1/2$-H\"older curve whose image contains a specified triangle. Its
normalized Fourier coefficients lie in $[-1,1]$, strictly within the
permitted interval $[-2,2]$. The available statements of Bober's result
establish the same qualitative conclusion; they do not provide the details
needed to compare the constructions or their quantitative properties.

The planar curves associated with Rudin--Shapiro sums are classical;
see Mend\`es France and Tenenbaum \cite{MFT} and Mend\`es France
\cite{MF}. We use two precise results on their lattice points, in the
form given by Rampersad and Shallit \cite{RS}.
The Fourier estimate needed here follows from the complementary-polynomial
identity for the Rudin--Shapiro sequence. It holds at every real frequency,
which allows a rotation and reflection to impose the symmetry in
\eqref{eq:support}.

\begin{theorem}[An explicit space-filling curve]\label{thm:curve}
There is an explicitly defined function $f\in\calS$ such that
\begin{enumerate}
 \item $f(0)=f(1)=0$ and $f(t)=-\overline{f(1-t)}$;
 \item $|f(t)-f(s)|\leq4|t-s|^{1/2}$ for $s,t\in[0,1]$;
 \item for every nonzero integer $h$,
 \begin{equation}\label{eq:strictcoeff}
   \alpha_h:=2\pi ih\,\widehat f(h)\in[-1,1];
 \end{equation}
 \item $f([0,1])$ contains the triangle
 \begin{equation}\label{eq:triangle}
   \Delta=\frac{\rho}{4}\{x+iy:0\leq y\leq x\leq\tfrac12\},
   \qquad \rho=\frac{1+i}{\sqrt2},
 \end{equation}
 of area $1/128$.
\end{enumerate}
\end{theorem}

The explicit definition is given in \eqref{eq:sewing} below.
The coefficient bound \eqref{eq:strictcoeff} is half the bound allowed
by \eqref{eq:support}. This margin is convenient in the
following general estimate. Its proof does not use space-filling geometry.

\begin{theorem}[A uniform small-ball lower bound]\label{thm:smallball}
For every $L\geq1$, there is a constant $C_L>0$ with the following
property. Suppose that $g\colon[0,1]\to\C$ is continuous,
$g(0)=g(1)=0$,
\begin{equation}\label{eq:generalclass}
 |g(t)-g(s)|\leq L|t-s|^{1/2},\qquad
 2\pi ih\,\widehat g(h)\in[-1,1]\quad(h\ne0).
\end{equation}
Then, for $0<\varepsilon\leq1/4$,
\begin{equation}\label{eq:smallball}
 \Prob\bigl(\norm{K-g}<\varepsilon\bigr)
 \geq\exp\bigl\{-C_L\varepsilon^{-2}\log^2(2/\varepsilon)\bigr\}.
\end{equation}
In particular, the constant is absolute when $g=f$ in
Theorem~\ref{thm:curve}.
\end{theorem}

We give an arithmetic consequence after the proofs. It uses the known
functional convergence of Kloosterman paths with both parameters varying;
no rate in the modulus is asserted.

\subsection*{Notation}
We write $\norm{v}=\sup_{0\leq t\leq1}|v(t)|$.
The notation $A\ll_L B$ means $|A|\leq C_LB$ with a constant depending
only on $L$; subscripts are omitted for absolute constants.
A curve is called space-filling here if its image has nonempty interior.

\section{The Rudin--Shapiro curve}

Define the Rudin--Shapiro sequence by
\begin{equation}\label{eq:rs}
 r(0)=1,\qquad r(2n)=r(n),\qquad r(2n+1)=(-1)^nr(n)
 \quad(n\geq0).
\end{equation}
We use partial sums with the upper endpoint excluded:
\begin{equation}\label{eq:partial}
 S(n)=\sum_{0\leq j<n}r(j),\qquad
 T(n)=\sum_{0\leq j<n}(-1)^jr(j),\qquad W(n)=S(n)+iT(n).
\end{equation}
For $m\geq0$, let $\gamma_m$ be the polygonal interpolation of
\begin{equation}\label{eq:vertices}
 \gamma_m(n/4^m)=2^{-m}W(n),\qquad 0\leq n\leq4^m.
\end{equation}

\begin{lemma}\label{lem:limit}
The sequence $(\gamma_m)$ converges uniformly to a continuous function
$\gamma$ satisfying $\gamma(0)=0$, $\gamma(1)=1+i$, and
\begin{equation}\label{eq:uniformerror}
 \norm{\gamma-\gamma_m}\leq\frac{\sqrt{10}}2\,2^{-m}.
\end{equation}
Moreover, $\gamma(n/4^m)=2^{-m}W(n)$ at every vertex of level $m$.
\end{lemma}

\begin{proof}
The recursion \eqref{eq:rs} gives
\begin{equation}\label{eq:fourrecursion}
 \begin{aligned}
 r(4n)&=r(n),& r(4n+1)&=r(n),\\
 r(4n+2)&=(-1)^nr(n),&r(4n+3)&=-(-1)^nr(n).
 \end{aligned}
\end{equation}
Consequently,
\begin{equation}\label{eq:scaling}
 S(4n)=2S(n),\qquad T(4n)=2T(n),\qquad W(4n)=2W(n).
\end{equation}
Thus consecutive interpolants agree at the older vertices, and
$W(4^m)=2^m(1+i)$.

To estimate the refinement error, put $a=r(n)$ and $b=(-1)^n$.
Relative to $W(4n)$, the four subsequent vertices of the unscaled finer
path are
\[
 a(1+i),\quad 2a,\quad a(2+b+ib),\quad 2a(1+ib).
\]
At the same four times the older chord has relative positions
$(j/2)a(1+ib)$, $1\leq j\leq4$.
The largest difference has modulus at most $\sqrt{10}/2$.
After rescaling, this gives
\[
 \norm{\gamma_{m+1}-\gamma_m}\leq\frac{\sqrt{10}}4\,2^{-m}.
\]
The difference is affine between successive finer vertices, so the
vertex estimate controls its uniform norm. Summing the resulting
geometric series proves \eqref{eq:uniformerror}. Agreement at older
vertices persists in the limit.
\end{proof}

\begin{lemma}\label{lem:holder}
For $s,t\in[0,1]$,
\begin{equation}\label{eq:holdergamma}
 |\gamma(t)-\gamma(s)|\leq8|t-s|^{1/2}.
\end{equation}
\end{lemma}

\begin{proof}
The modulus of the slope of $\gamma_m$ is $\sqrt2\,2^m$.
For $\delta=|t-s|>0$, choose $m\geq0$ so that
$4^{-(m+1)}<\delta\leq4^{-m}$.
Lemma~\ref{lem:limit} then gives
\[
 |\gamma(t)-\gamma(s)|
 \leq\sqrt2\,2^m\delta+\sqrt{10}\,2^{-m}
 \leq(\sqrt2+2\sqrt{10})\sqrt\delta<8\sqrt\delta.
\]
\end{proof}

We next recall exactly which lattice-point facts are used. In the
inclusive convention
\[
 s_*(n)=\sum_{j=0}^nr(j),\qquad
 t_*(n)=\sum_{j=0}^n(-1)^jr(j),
\]
Rampersad and Shallit \cite[Theorem~31]{RS} prove that every nonzero
integer pair $(X,Y)$ with $X\geq Y\geq0$ and $X\equiv Y\pmod2$ occurs as
$(s_*(n),t_*(n))$. The Brillhart--Morton estimate
\cite{BM}, stated as \cite[Theorem~18]{RS}, is
\begin{equation}\label{eq:lowerRS}
 s_*(n)>\sqrt{3n/5}\qquad(n\geq1).
\end{equation}
Both statements are used as known results.

\begin{lemma}\label{lem:fill}
The image of $\gamma$ contains
\begin{equation}\label{eq:unrotatedtriangle}
 \{x+iy:0\leq y\leq x\leq\tfrac12\}.
\end{equation}
\end{lemma}

\begin{proof}
Fix $0<y<x<1/2$. Choose positive integers $X_m>Y_m$ of the same parity
with $2^{-m}(X_m+iY_m)\to x+iy$.
By the lattice-point result there is an index $k_m$ such that
$(s_*(k_m),t_*(k_m))=(X_m,Y_m)$.
For all sufficiently large $m$, \eqref{eq:lowerRS} implies
\[
 \frac{k_m+1}{4^m}
 \leq\frac53\left(\frac{X_m}{2^m}\right)^2+4^{-m}<1.
\]
Indeed, the upper bound tends to $5x^2/3<5/12$.
The shift by one converts the inclusive sums to \eqref{eq:partial}.
By Lemma~\ref{lem:limit},
\[
 \gamma\left(\frac{k_m+1}{4^m}\right)
 =2^{-m}(X_m+iY_m).
\]
Since $\gamma([0,1])$ is compact, it contains $x+iy$ and hence the
closure of the set of all such points. This is
\eqref{eq:unrotatedtriangle}.
\end{proof}

\section{Fourier bounds and reflection}

For $q\geq0$, write
\[
 P_q(z)=\sum_{0\leq n<2^q}r(n)z^n.
\]
We use the classical complementary-polynomial identity associated with
the Rudin--Shapiro sequence \cite{Rudin}; its short proof fixes the
normalization and the particular pair of polynomials.

\begin{lemma}\label{lem:complement}
For $q\geq0$ and $|z|=1$,
\begin{equation}\label{eq:complement}
 |P_q(z)|^2+|P_q(-z)|^2=2^{q+1}.
\end{equation}
\end{lemma}

\begin{proof}
Separating the even and odd indices in \eqref{eq:rs} gives
\[
 P_{q+1}(z)=P_q(z^2)+zP_q(-z^2),\qquad
 P_{q+1}(-z)=P_q(z^2)-zP_q(-z^2).
\]
The sum of the squares of their moduli is twice the corresponding
sum at level $q$. At level zero it equals $2$.
\end{proof}

\begin{proposition}\label{prop:transform}
For every $m\geq0$ and $\xi\in\R$,
\begin{equation}\label{eq:transformbound}
 A_m(\xi):=\int_0^1\gamma_m'(u)e(-\xi u)\,du
 \quad\text{satisfies}\quad |A_m(\xi)|\leq2.
\end{equation}
\end{proposition}

\begin{proof}
Put $N=4^m$. On the interval $(n/N,(n+1)/N)$,
\[
 \gamma_m'(u)=\sqrt N\,[r(n)+i(-1)^nr(n)].
\]
For $z=e(-\xi/N)$ and
\[
 D(v)=\int_0^1e(-vs)\,ds
     =e(-v/2)\frac{\sin(\pi v)}{\pi v},\qquad D(0)=1,
\]
integration on each interval yields the exact formula
\begin{equation}\label{eq:transformformula}
 A_m(\xi)=\frac{D(\xi/N)}{\sqrt N}
          \bigl(P_{2m}(z)+iP_{2m}(-z)\bigr).
\end{equation}
By \eqref{eq:complement} and Cauchy--Schwarz,
\[
 |P_{2m}(z)+iP_{2m}(-z)|
 \leq\sqrt2\bigl(|P_{2m}(z)|^2+|P_{2m}(-z)|^2\bigr)^{1/2}
 =2\sqrt N.
\]
Since $|D(v)|\leq1$, the proposition follows.
\end{proof}

\begin{proof}[Proof of Theorem~\ref{thm:curve}]
Let $\rho=(1+i)/\sqrt2$ and define
\begin{equation}\label{eq:sewing}
 f(t)=
 \begin{cases}
  \dfrac\rho4\gamma(2t),&0\leq t\leq1/2,\\[4pt]
  -\dfrac{\bar\rho}4\overline{\gamma(2-2t)},&1/2\leq t\leq1.
 \end{cases}
\end{equation}
The two values at $t=1/2$ agree because
$\rho(1+i)=i\sqrt2=-\bar\rho(1-i)$.
The endpoint and symmetry assertions follow immediately.
Lemma~\ref{lem:holder} gives a H\"older constant $2\sqrt2$ on each
half-interval. If $s\leq1/2\leq t$, split at $1/2$ and use
$\sqrt{1/2-s}+\sqrt{t-1/2}\leq\sqrt{2(t-s)}$.
Thus the constant on $[0,1]$ is at most $4$.
Lemma~\ref{lem:fill} shows that the first half of the image contains
$\Delta$, whose area is $(1/16)(1/8)=1/128$.

Let $f_m$ be defined by \eqref{eq:sewing} with $\gamma_m$ in place of
$\gamma$. These functions are continuous, piecewise affine, closed, and
converge uniformly to $f$.
For $h\in\Z$, changing variables on the two halves gives
\begin{align}
 \int_0^1f_m'(t)e(-ht)\,dt
 &=\tfrac14\rho A_m(h/2)
   +\tfrac14\bar\rho\,\overline{A_m(h/2)}\notag\\
 &=\tfrac12\Rea\bigl(\rho A_m(h/2)\bigr)\in[-1,1].
 \label{eq:sewntransform}
\end{align}
Here $e(-h)=1$ is used on the second half. Proposition~\ref{prop:transform}
applies also when $h/2$ is not integral.
For $h\ne0$, integration by parts, with vanishing boundary values, gives
\[
 2\pi ih\,\widehat f_m(h)=\int_0^1f_m'(t)e(-ht)\,dt\in[-1,1].
\]
Uniform convergence allows passage to the limit in each Fourier
coefficient. This proves \eqref{eq:strictcoeff} without any assertion
about differentiability of $f$ or convergence of the derivatives.
Finally, \eqref{eq:support} implies $f\in\calS$.
\end{proof}

\begin{remark}
Theorem~\ref{thm:curve} constructs an element of the support with
nonempty interior in its image. It does not give such a property for
almost every realization of $K$.
\end{remark}

\section{A lower bound for small balls}

We first give the tail estimate used in the proof of
Theorem~\ref{thm:smallball}. Including incomplete blocks also makes the
uniform-convergence argument independent of the symmetric truncation point.

\begin{lemma}\label{lem:tail}
There is an absolute constant $C>0$ such that, for every integer $M\geq1$,
\begin{equation}\label{eq:tail}
 \E\norm{\sum_{|h|>M}\frac{e(ht)-1}{2\pi ih}X_h}
 \leq C\sqrt{\frac{\log(2M)}M}.
\end{equation}
The symmetric series defining the tail converges almost surely
uniformly and in expected uniform norm.
\end{lemma}

\begin{proof}
For integers $L\geq1$ and $L\leq H\leq2L$, put
\[
 Z_{L,H}(t)=\sum_{L<|h|\leq H}\frac{X_h}{2\pi ih}e(ht).
\]
Since the variables are independent, centred and bounded by $2$,
Hoeffding's inequality, applied to the real and imaginary parts, gives
\begin{equation}\label{eq:hoeffding}
 \Prob(|Z_{L,H}(t)|>v)\leq4\exp(-cLv^2)
\end{equation}
with an absolute $c>0$, uniformly in $H$ and $t$.
Indeed, the sum of the squared coefficient bounds is $O(L^{-1})$.

Each polynomial has degree at most $2L$ and satisfies Bernstein's
inequality
$\norm{Z_{L,H}'}\leq4\pi L\norm{Z_{L,H}}$.
A grid of $\lceil16\pi L\rceil$ equally spaced points on $\R/\Z$
therefore controls the uniform norm within a factor of $2$.
Taking a union bound in \eqref{eq:hoeffding} over this grid and over
the $L+1$ possible values of $H$ gives
\[
 \Prob\left(\max_{L\leq H\leq2L}\norm{Z_{L,H}}>2v\right)
 \leq \min\{1,CL^2\exp(-cLv^2)\}.
\]
Integrating this estimate yields
\begin{equation}\label{eq:incomplete}
 \E\max_{L\leq H\leq2L}\norm{Z_{L,H}}
 \ll\sqrt{\frac{\log(2L)}L}.
\end{equation}
The extra union over $H$ changes only the absolute constant.

For the process increments, subtract $Z_{L,H}(0)$; this multiplies the
bound by at most $2$.
The sum of the expectations of these maximal block norms, for
$L=M,2M,4M,\ldots$, is
$O(\sqrt{\log(2M)/M})$.
It follows that their sum is finite almost surely, so full blocks
converge uniformly and the incomplete blocks tend uniformly to zero.
The same summation proves \eqref{eq:tail} and convergence in expected
uniform norm.
\end{proof}

\begin{proof}[Proof of Theorem~\ref{thm:smallball}]
Set $\beta_h=2\pi ih\,\widehat g(h)\in[-1,1]$ for $h\ne0$.
Let $\sigma_M g$ denote the Fej\'er mean of order $M$, and put
\begin{equation}\label{eq:fejer}
 q_M(t)=\sigma_Mg(t)-\sigma_Mg(0)
 =\sum_{0<|h|\leq M}\left(1-\frac{|h|}{M+1}\right)
          \beta_h\frac{e(ht)-1}{2\pi ih}.
\end{equation}
Because $g(0)=g(1)=0$, its periodic extension has modulus of continuity
at most $\sqrt2L\sqrt u$ at circular distance $u\leq1/2$.
Indeed, an interval crossing the endpoint is split there, and
$\sqrt a+\sqrt b\leq\sqrt{2(a+b)}$.
The usual bound for the Fej\'er kernel is
\[
 F_M(u)\ll\min\left(M+1,\frac1{(M+1)u^2}\right),
 \qquad0<|u|\leq1/2.
\]
Integrating $L\sqrt{|u|}F_M(u)$, and splitting at $|u|=(M+1)^{-1}$,
gives
\begin{equation}\label{eq:fejererror}
 \norm{q_M-g}\ll L M^{-1/2}.
\end{equation}

Write $\ell=\log(2/\varepsilon)$ and choose
\[
 M=\left\lceil A_L\varepsilon^{-2}\ell\right\rceil
\]
with $A_L$ sufficiently large, depending only on $L$.
Then \eqref{eq:fejererror}, Lemma~\ref{lem:tail}, and Markov's inequality
give
\begin{equation}\label{eq:threeevents}
 \norm{q_M-g}<\varepsilon/3,\qquad
 \Prob(\norm{R_M}<\varepsilon/3)\geq1/2,
 \quad
 R_M(t)=\sum_{|h|>M}\frac{e(ht)-1}{2\pi ih}X_h.
\end{equation}
For completeness, $\log(2M)\ll(1+\log A_L)\ell$, uniformly for
$0<\varepsilon\leq1/4$, so increasing $A_L$ makes the expected norm
of $R_M$ smaller than $\varepsilon/6$.

Set $b_0=0$ and
$b_h=(1-|h|/(M+1))\beta_h$ for $0<|h|\leq M$. All the $b_h$ lie in
$[-1,1]$. Put
\[
 \eta=\frac{\varepsilon}{12(1+\log M)},\qquad
 E_M=\bigcap_{|h|\leq M}\{|X_h-b_h|<\eta\}.
\]
On $E_M$, the low-frequency part of $K$ differs from $q_M$ by at most
\begin{equation}\label{eq:lowerror}
 \eta\left(1+\frac2\pi\sum_{h=1}^M\frac1h\right)
 \leq2\eta(1+\log M)=\varepsilon/6.
\end{equation}
The Sato--Tate density is bounded below on $[-5/4,5/4]$.
In particular, for $b\in[-1,1]$ and $0<\eta\leq1/4$,
\[
 \Prob(|X_h-b|<\eta)\geq\eta/3.
\]
Independence gives
$\Prob(E_M)\geq(\eta/3)^{2M+1}$.
Moreover, $E_M$ is independent of the tail event in
\eqref{eq:threeevents}, which uses only $|h|>M$.
Combining \eqref{eq:threeevents} and \eqref{eq:lowerror}, we obtain
\[
 \Prob(\norm{K-g}<\varepsilon)
 \geq\tfrac12(\eta/3)^{2M+1}.
\]
Finally, $M\ll_L\varepsilon^{-2}\ell$ and
$\log(3/\eta)\ll_L\ell$. This proves \eqref{eq:smallball}.
\end{proof}

\begin{remark}
The proof is a finite-coefficient approximation argument with a
subgaussian estimate for the remaining series. No optimality is claimed
for the exponent or the logarithmic factor in \eqref{eq:smallball}.
\end{remark}

\section{Approximation by arithmetic Kloosterman paths}

For a prime $p$ and $(a,b)\in(\F_p^\times)^2$, let
$K_p(\cdot;a,b)$ be the polygonal interpolation of
\begin{equation}\label{eq:arithvertices}
 K_p\left(\frac j{p-1};a,b\right)
 =\frac1{\sqrt p}\sum_{1\leq x\leq j}
     \ee{\frac{ax+b\bar x}{p}},\qquad0\leq j\leq p-1,
\end{equation}
where $\bar x$ is the inverse of $x$ modulo $p$.
When both parameters are uniform and nonzero,
\cite[Theorem~1.5]{KSpaths} gives
$K_p\Rightarrow K$ in $C([0,1],\C)$.
The parameters in that theorem are related to those above by the
bijection $(\alpha,c)\mapsto(a,b)=(\alpha c,\alpha)$.

\begin{corollary}\label{cor:arithmetic}
For the function $f$ of Theorem~\ref{thm:curve}, there is an absolute
$C>0$ such that, for $0<\varepsilon\leq1/4$,
\begin{equation}\label{eq:arithprob}
 \liminf_{\substack{p\to\infty\\p\text{ prime}}}\frac1{(p-1)^2}
 \#\{(a,b)\in(\F_p^\times)^2:
          \norm{K_p(\cdot;a,b)-f}<\varepsilon\}
 \geq\exp\{-C\varepsilon^{-2}\log^2(2/\varepsilon)\}.
\end{equation}
\end{corollary}

\begin{proof}
The ball $\{v:\norm{v-f}<\varepsilon\}$ is open. The Portmanteau
theorem and functional convergence give a lower bound for the left-hand
side by $\Prob(\norm{K-f}<\varepsilon)$.
Apply Theorem~\ref{thm:smallball} with $L=4$.
\end{proof}

Let $V_p(a,b)$ denote the set of vertices in \eqref{eq:arithvertices}.
We say that it is $\varepsilon$-dense in $\Delta$ if
\[
 \sup_{z\in\Delta}\inf_{v\in V_p(a,b)}|z-v|<\varepsilon.
\]

\begin{corollary}\label{cor:dense}
There is an absolute constant $C>0$ such that, for
$0<\varepsilon\leq1/4$, the limiting lower proportion of pairs $(a,b)$
for which $V_p(a,b)$ is $\varepsilon$-dense in $\Delta$ and
$|K_p(1;a,b)|<\varepsilon$ is at least
\[
 \exp\{-C\varepsilon^{-2}\log^2(2/\varepsilon)\}.
\]
\end{corollary}

\begin{proof}
Suppose that $\norm{K_p-f}<\varepsilon/2$.
For $z\in\Delta$, choose $t$ with $f(t)=z$.
Then $|K_p(t)-z|<\varepsilon/2$.
Each segment of $K_p$ has length $p^{-1/2}$, so a vertex lies within
$\varepsilon/2+p^{-1/2}$ of $z$.
This is less than $\varepsilon$ once $p>4\varepsilon^{-2}$.
Also $|K_p(1)|<\varepsilon/2$, since $f(1)=0$.
Apply Corollary~\ref{cor:arithmetic} with $\varepsilon/2$, changing the
absolute constant.
\end{proof}

In both corollaries $\varepsilon$ is fixed before $p$ tends to infinity.
The statements concern the two-parameter family. They provide no
effective threshold for $p$ as a function of $\varepsilon$.

\end{document}